\documentclass[11pt]{amsart}
\usepackage{amsmath,amssymb,amsfonts}
\usepackage{enumerate}
\usepackage{mathptmx}
\usepackage{eulervm}
\usepackage{tikz}
\usepackage{hyperref}
\hypersetup{colorlinks=true, citecolor=magenta, linkcolor=blue}

\newtheorem{thm}{Theorem}

\newtheorem{theorem}{Theorem}

\newtheorem{proposition}{Proposition}
\newtheorem{lemma}[proposition]{Lemma}
\newtheorem{corollary}[proposition]{Corollary}
\theoremstyle{definition}
\newtheorem{remark}[proposition]{Remark}

\newcommand{\mtwo}[4]{\begin{bmatrix}#1&#2\\#3&#4\end{bmatrix}}
\newcommand{\fr}[2]{^#1\!\!/\!_#2}

\newcommand{\op}[1]{\operatorname{#1}}
\newcommand{\ce}[1]{\widetilde{#1}}
\newcommand{\tp}[1]{{^\intercal\:\!\!#1}}
\newcommand{\tpinv}[1]{{^\intercal\:\!\!#1^{-1}}}

\newcommand{\GL}{\operatorname{GL}}
\newcommand{\SL}{\operatorname{SL}}
\newcommand{\Sp}{\operatorname{Sp}}
\newcommand{\val}{\operatorname{val}}
\newcommand{\vol}{\operatorname{vol}}
\newcommand{\aff}{\operatorname{aff}}
\newcommand{\ch}{\mathcal{X}}

\newcommand{\Z}{\mathbf{Z}}
\newcommand{\C}{\mathbf{C}}
\newcommand{\R}{\mathbf{R}}
\newcommand{\kk}{\mathbf{k}}

\newcommand{\h}{\mathfrak{h}}
\newcommand{\g}{\mathfrak{g}}
\newcommand{\oo}{\mathfrak{o}}
\newcommand{\XX}{\mathfrak{X}}
\newcommand{\spp}{\mathfrak{sp}}

\newcommand{\xx}{\widetilde{x}}
\newcommand{\hh}{\widetilde{h}}
\newcommand{\ww}{\widetilde{w}}

\newcommand{\KK}{\widetilde{K}}
\newcommand{\HH}{\widetilde{H}}
\newcommand{\II}{\widetilde{I}}

\newcommand{\LL}{\mathcal{L}}

\title
[On the metaplectic group in even residual characteristic]
{On the metaplectic group in even residual characteristic}

\author[G.\ Savin]{Gordan Savin}

\author[A.\ Wood]{Aaron Wood}

\begin{document}
\maketitle

\begin{abstract}
For maximal compact subgroups of the metaplectic group, the minimal types in the Schr\"odinger model of the Weil representation are calculated explicitly. Although these types are known in the case of odd residual characteristic, this computation is done for arbitrary residual characteristic.
\end{abstract}

%\setcounter{tocdepth}{1}
%\tableofcontents

%%%%%%%%%%
%%%%%%%%%%
\section*{Introduction}
\label{introduction}
%%%%%

Let $\kk$ be a nonarchimedian local field, $\oo$ its ring of integers, and $\varpi$ a uniformizer of $\oo$;  the characteristic of $\kk$ is assumed to not be 2; there is no restriction on the residual characteristic, i.e., $q = |\oo / \varpi \oo|$ may be any prime power.  Fix a nontrivial, smooth, additive character $\psi$ of $\kk$ of conductor $2e$, where $e$ is the valuation of 2 in $\kk$.  Let $W$ be a symplectic vector space over $\kk$ and $\omega = \omega_\psi$ the Weil representation of the metaplectic group $\ce \Sp(W)$.

The restriction of $\omega$ to the maximal compact subgroups of $\ce\Sp(W)$ is well understood if $\kk$ has odd residual characteristic, for instance, in \cite{prasad}.  The goal of this paper is to extend this understanding to the case of even residual characteristic.

Fix a symplectic basis 
	$\{ 
		e_1, \dots, e_n, f_1, \dots, f_n 
	\}$ 
of $W$; let $Y$ be the subspace spanned by the $f_i$.  The standard lattice $\LL_0 \subset W$ is the $\oo$-span of the symplectic basis.  Define the sequence of lattices
	$\LL_0 \supset \LL_1 \supset \dots \supset \LL_n$,
where $\LL_i$ is the $\oo$-span of 
	$\{ 
		e_1, \dots, e_n, 
		\varpi f_1, \dots, \varpi f_i, f_{i+1}, \dots, f_n 
	\}$.  
Let $\LL^\ast_i$ be the dual lattice, i.e., the $\oo$-span of 
	$\{
		\varpi^{-1} e_1, \dots, \varpi^{-1} e_i, e_{i+1}, \dots, e_n, 
		f_1, \dots, f_n
	\}$.

Let $K_i$ be the subgroup of $\Sp(W)$ which stabilizes $\LL_i$ (equivalently $\LL^\ast_i$).  The groups $K_0$, \dots, $K_n$ are the only maximal compact subgroups which contain the Iwahori subgroup $I = \cap K_i$, and every maximal compact subgroup of $\Sp(W)$ is conjugate to one of these $K_i$.  For each $i$, consider the lattices 
	$L_i 
	=	\LL_i \cap Y$ 
and 
	$L'_i 
	=	\LL^\ast_i \cap Y$.  
Note that $L'_i = L_0$ for all $i$.

Let $S(Y)$ be the space of Schwartz functions on $Y$.  The lattices $L_i$ and $L'_i$ induce a filtration on $S(Y)$ as follows.  For an integer $m$, let $S_{i,m}$ be the space of Schwartz functions which are supported on $\varpi^{-m} L'_i$ and invariant under translation by $2 \varpi^m L_i$, and let $S'_{i,m}$ be the space of those which are supported on $\varpi^{-m} L_i$ and invariant under translation by $2 \varpi^m L'_i$.  Each of these may be identified with the Schwartz functions on a quotient space:
\[	
	S_{i,m} 
	=	S( \varpi^{-m} L'_i / 2 \varpi^m L_i )
	\quad \text{ and } \quad
	S'_{i,m} 
	=	S( \varpi^{-m} L_i / 2 \varpi^m L'_i ).
\] 
For $U \subset S(Y)$, denote by $U^+$ (resp. $U^-$) the even (resp. odd) functions in $U$.

Let $\KK_i$ denote the full inverse image of $K_i$ in $\ce \Sp(W)$.  The main theorem of this paper describes the minimal types of $\KK_i$ in the Schr\"odinger model of the Weil representation $\omega$ on $S(Y)$.

%%%%%
\begin{thm}
\label{main_theorem}
The maximal compact subgroup $\KK_i$ preserves the filtration
\[
	S_{i,0} \subset 
	S'_{i,1} \subset
	S_{i,1} \subset 
	S'_{i,2} \subset
	S_{i,2} \subset \dots
\]
and acts irreducibly on the components $S_{i,0}^+$ and $S_{i,0}^-$ of the minimal type.
\end{thm}
%%%%%

To describe these types more explicitly, denote the $\oo$-span of any $r$ basis elements of $Y = \kk^n$ by $\oo^r$.  That is, $L_i = \varpi \oo^i \oplus \oo^{n-i}$ and $L'_i = \oo^n$.

For $i = 0$, the filtration is
	$S_{0,0} \subset S_{0,1} \subset S_{0,2} \subset \dots$
and the minimal type of $\KK_0$ is 
\[
	S_{0,0} 
	=	S^+_{0,0} 
	=	S( L_0 / 2 L_0 ) 
	=	S( \oo^n / 2 \oo^n).
\]
The dimension of $S_{0,0}$ is $q^{en}$.

For $i = n$, the filtration is
	$S_{n,0} \subset S_{n,1} \subset S_{n,2} \subset \dots$
and the minimal type of $\KK_n$ decomposes as 
	$S_{n,0} 
	=	S^+_{n,0} \oplus S^-_{n,0}$, 
where
\[
	S_{n,0}^\pm 
	=	S( L'_n / 2 L_n )^\pm 
	=	S( \oo^n / 2 \varpi \oo^n )^\pm.
\]
The dimension of $S_{n,0}^\pm$ is $\tfrac{1}{2} q^{en} (q^n \pm 1)$.

Generally, the minimal type of $\KK_i$ decomposes as 
	$S_{i,0} 
	=	S^+_{i,0} \oplus S^-_{i,0}$, 
where
\[
	S_{i,0}^\pm 
	=	S( L'_i / 2 L_i )^\pm 
	=	S( \oo^i / 2 \varpi \oo^i )^\pm 
			\otimes 
		S( \oo^{n-i} / 2 \oo^{n-i} ).	
\]
The dimension of $S_{i,0}^\pm$ is $\frac{ q^{en} }{2} (q^i \pm 1)$.
  
%A remark on an application is in order.  In \cite{gan-savin}, the first author (with Gan) made use of the minimal types (in odd residual characteristic) to construct Hecke algebra correspondences between the metaplectic group and the orthogonal group.  Later, in \cite{wood}, the second author employed a different method to give one of the same correspondences over the field $\Q_2$; unfortunately, this method is inadequate for ramified extensions of $\Q_2$.  However, the current description of minimal types in any residual characteristic allows for the result of Gan and the first author to include the 2-adic fields, the details of which are worked out by Takeda and the second author in \cite{takeda-wood}.

The first six sections are dedicated to providing the necessary background material and notation.  Theorem \ref{main_theorem} is proved in \S\ref{K-types} as Theorems \ref{K_0-theorem}, \ref{K_n-theorem}, and \ref{K_i-theorem}.

%%%%%%%%%%
%%%%%%%%%%
\section*{Notation}
\label{notation}
%%%%%

Throughout this paper, let $\kk$ be a nonarchimedian local field; assume the characteristic of $\kk$ is different from 2.  Let $\oo$ be the ring of integers and $\varpi$ a uniformizer of $\oo$.  Denote by $q$ the order of the residue field $\oo / \varpi \oo$.

Let $e$ be the valuation of 2 in $\kk$: if $q$ is even, then $e$ is the ramification index of 2; if $q$ is odd, then $e = 0$. 

Fix a nontrivial, smooth, additive character $\psi$ of $\kk$; assume that the conductor of $\psi$ is $c = 2e$, i.e., that $\psi(t x) = 1$ for all $t \in \oo$ if and only if $x \in 4\oo$.

For a $2n$-dimensional symplectic vector space over any field, fix a symplectic basis $\{ e_1, \dots, e_n, f_1, \dots, f_n\}$ so that the symplectic form $Q$ is given by
\[
	Q(u,v) 
	=	\tp{u} \mtwo{0}{1}{-1}{0} v.
\]
Let $X$ be the subspace spanned by $\{e_1, \dots, e_n\}$ and $Y$ the subspace spanned by $\{f_1,\dots, f_n\}$; the decomposition $W = X + Y$ is a polarization of the symplectic space.

%%%%%%%%%%
%%%%%%%%%%
\section{Schwartz functions}
\label{schwartz_functions}
%%%%%

Let $V$ be a finite-dimensional vector space over $\kk$ with a Haar measure $dv$.  A Schwartz function on $V$ is a smooth, compactly supported, complex-valued function; the space of such functions will be designated by $S(V)$.

Let 
	$V^\ast 
	=	\op{Hom}_\kk(V, \kk)$ 
be the linear dual of $V$.  Define the Fourier transform, relative to $\psi$, to be the map from $S(V)$ to $S(V^\ast)$, written 
	$\phi \mapsto \widehat{\phi}$, 
given by
\[
	\widehat{\phi}(v^\ast)
	=	\int_V \psi(2 \langle v, v^\ast \rangle) \phi(v) dv.
\]
Here, $\langle\;\,,\;\rangle$ is the canonical pairing on $V \times V^\ast$.  The Haar measure $dv^\ast$ on $V^\ast$ is normalized so that the Fourier transform is self-dual:
\[
	\widehat{\widehat{\phi}}(-v) 
	=	\phi(v).
\]
	
By fixing a basis of $V$ and identifying $v^\ast \in V^\ast$ with $v \in V$ by 
	$\langle u, v^\ast \rangle 
	=	\tp{u} v$, 
the Fourier transform becomes the operator on $S(V)$ given by
\[
	\widehat{\phi}(u) 
	=	\int_V \psi(2 \tp{v} u) \phi(v) dv.	
\]
The identification of $V$ and $V^\ast$ gives a unique normalization of the Haar measure $dv$.  Moreover, $S(V)$  has the tensor product structure 
	$S(\kk) \otimes \cdots \otimes S(\kk)$, 
and the Fourier transform on $S(V)$ is precisely the Fourier transform on each of the $S(\kk)$ factors.

Fix the lattice $L$ to be the $\oo$-span of the chosen basis of $V$.  For each $\phi \in S(V)$, there exist integers $r,s$ with $r \leq e + s$, such that $\phi$ is supported on $\varpi^r L$ and invariant under translation by $2 \varpi^s L$.  The subspace of such functions will be identified with the space of Schwartz functions on 
	$\varpi^r L / 2 \varpi^s L$, 
of dimension $q^{(e + s - r) n}$.

Note that $v \in V$ is congruent to $-v$ modulo $2 \varpi^s L$ if and only if $v \in \varpi^s L$; that is, for $r \geq s$, there are no odd functions in 
	$S(\varpi^r L / 2 \varpi^s L)$.  
Otherwise, it decomposes into subspaces of even and odd functions, denoted respectively by 
	$S(\varpi^r L / 2 \varpi^s L)^+$ 
and 
	$S(\varpi^r L / 2 \varpi^s L)^-$, 
of dimensions
\[
	\dim S(\varpi^r L / 2 \varpi^s L)^\pm 
	=	\tfrac{1}{2} q^{en} \left(q^{(s - r) n} \pm 1 \right).	
\]
	
To be explicit, for $U \subset V$, designate the characterstic function on $U$ by $\ch(U)$.  For a fixed $s \in \Z$, if $v \in \varpi^s L$, define the even function 
	$\phi_v^+ 
	=	\ch(v + 2 \varpi^s L)$; 
otherwise, define the even function 
\[
	\phi_v^{+} 
	=	\ch(v + 2 \varpi^s L) + \ch(-v + 2 \varpi^s L)
\]
and the odd function
\[
	\phi_v^- 
	=	\ch(v + 2 \varpi^s L) - \ch(-v + 2 \varpi^s L).
\]
(Note that only one of $\{v, -v\}$ is needed in order to define a set of linearly independent functions.)  The functions $\phi_v^\pm$, as $v$ ranges over 
	$\varpi^r L / 2 \varpi^s L$ (modulo $\pm1$), 
form a basis for 
	$S(\varpi^r L / 2 \varpi^s L)^\pm$, 
giving the dimension count above.

%%%%%
\begin{lemma}
\label{fourier_transform_space}
For any integers $r,s$ with $r \leq e + s$, the Fourier transform maps 
\[
	S(\varpi^r L / 2 \varpi^s L) 
		\to
	S(\varpi^{-s} L / 2 \varpi^{-r} L).		
\]
Moreover, the Fourier transform of 
	$\ch( 2 \varpi^s L )$ 
is 
	$\vol( 2 \varpi^s L ) \ch( \varpi^{-s} L )$.		
\end{lemma}

\begin{proof}
Suppose that $\phi$ is invariant under translation by any $z \in 2 \varpi^s L$.  Then,
\[
	\widehat{\phi}(v)
	=	\int_V \psi( 2 \tp{u} v ) \phi(u) du 
	=	\int_V \psi( 2 \tp{u} v ) \phi(u + z) du 
	=	\psi( -2 \tp{z} v ) \widehat{\phi}(v).
\]
Hence, either 
	$\widehat{\phi}(v) 
	=	0$, 
or 
	$\psi( -2 \tp{z} v ) 
	=	1$ 
for all $z \in 2 \varpi^s L$; that is, the support of $\widehat\phi$ is contained in $\varpi^{-s} L$.  Applying this to $\widehat\phi$ and using the self-duality of the Fourier transform, the first statement is proved.  As
\[
	\widehat{\ch( 2 \varpi^s L} ) (v) 
	=	\int_{2 \varpi^s L} \psi(-2 \tp{u} v ) du
	= 	\begin{cases}
			\vol(2 \varpi^s L) 
				& \text{ if } v \in \varpi^{-s} L 
			\\
			0 	
				& \text{ otherwise},
		\end{cases}
\]
the second statement is also proved.
\end{proof}
%%%%%

%%%%%%%%%%
%%%%%%%%%%
\section{Roots and affine roots}
\label{roots_and_affine_roots}
%%%%%

Let $\g$ be a simple Lie agebra over $\C$, $\h$ a Cartan subalgebra, and 
	$\h^\ast 
	=	\op{Hom}_\C( \h, \C)$ 
its linear dual.  The natural pairing 
	$(\;\,,\;): \h \times \h^\ast \to \C$ 
is given by
\[
	(a, \alpha) 
	=	\alpha(a).
\]
Designate by $\Delta \subset \h^\ast$ the roots of $\h$ and choose a subset $\Pi$ of simple roots; write $\Delta^+$ and $\Delta^-$ for the corresponding sets of positive and negative roots, respectively.  For each root $\alpha \in \h^\ast$, define the co-root $\check\alpha \in \h$ by 
	$(\check\alpha,\alpha)=2$. 

For $\alpha \in \h^\ast$ and $m \in \Z$, define the affine functional 
	$\alpha + m: 
		\h \to \C$ 
by
\[
	(\alpha+m)(a) 
	=	\alpha(a) + m.
\]
The set of affine roots is defined to be 
	$\Delta^{\aff} 
	=	\{ 
			\alpha + m 
			: \alpha \in \Delta, m \in \Z 
		\}$.  
If $\delta$ is the highest root in $\Delta$, then 
	$\Pi^{\aff} 
	=	\Pi \cup \{-\delta + 1\}$ 
is a set of simple affine roots.

Let $\h_\R$ be the real vector space spanned by the co-roots. Each 
	$\alpha + m 
	\in	\Delta^{\aff}$ 
determines an affine hyperplane
\[
	P_{\alpha+m} 
	=	\big\{ 
			a \in \h_\R 
			: (\alpha+m)(a) = 0 
		\big\},
\]
and a reflection $s_{\alpha+m}$ of $\h_\R$ across $P_{\alpha+m}$, i.e., 
\[
	s_{\alpha+m}(a) 
	=	a - (\alpha+m)(a) \check\alpha.
\]
For $d \in \h$, let 
	$T(d) : \h \to \h$ 
be the translation $T(d)(a) = a + d$.  In particular,
\[
	s_{\alpha+m} 
	=	T(-m \check\alpha) s_\alpha.
\]
	
The Weyl group $\Omega$ is the group generated by the reflections 
	$\{
		s_\alpha 
		: \alpha \in \Delta 
	\}$; 
similarly, the affine Weyl group $\Omega^{\aff}$ is generated by the affine reflections 
	$\{
		s_\mu 
		: \mu \in \Delta^{\aff} 
	\}$.  
In fact, $\Omega$ is generated by simple reflections 
	$\{
		s_\alpha 
		: \alpha \in \Pi
	\}$ 
and $\Omega^{\aff}$ is generated by simple affine reflections 
	$\{
		s_\mu 
		: \mu \in \Pi^{\aff} 
	\}$.
	
Define the co-root lattice in $\h_\R$ to be the $\Z$-span of the simple co-roots, i.e., the $\Z$-span of 
	$\{ 
		\check\alpha 
		: \alpha \in \Pi 
	\}$, 
and let $D$ be the group of translations by elements of the co-root lattice.  Since 
	$s_{-\delta+1} 
	=	T(\check\delta) s_{-\delta}$ 
and 
	$w T(d) w^{-1} 
	=	T\big( w(d) \big)$ 
for all $w \in \Omega$, the affine Weyl group $\Omega^{\aff}$ is the semidirect product $D \Omega$.

The chambers of $\Delta^{\aff}$ are the connected components of the complement in $\h_\R$ to the collection of hyperplanes 
	$\{
		P_\mu 
		: \mu \in \Delta^{\aff}
	\}$.
As the affine Weyl group permutes the collection of affine hyperlanes, it acts on the set of chambers; this action is simply transitive.  The fundamental chamber is the open set
\[
	\mathcal{C}_0 
	= 	\big\{ 
			a \in \h 
			: 0 < (a, \mu) < 1 \text{ for all } \mu \in \Pi^{\aff} 
		\big\},
\]
which is bounded by the hyperplanes 
	$\{ P_\mu : \mu \in \Pi^{\aff} \}$.  
See \cite[\S1.3]{iwahori-matsumoto} for details.

%%%%%%%%%%
%%%%%%%%%%
\section{Symplectic Lie algebras}
\label{symplectic_lie_algebras}
%%%%%

Let 
	$W_\C 
	=	X_\C + Y_\C$ 
be a $2n$-dimensional complex symplectic vector space. The symplectic Lie algebra $\spp(W_\C)$ is the algebra of endomorphisms $T$ of $W_\C$ satisfying $Q(Tu,v) + Q(u,Tv) = 0$ for all $u,v \in W_\C$.  

With respect to the symplectic basis, the symplectic Lie algebra is the subalgebra of $2n \times 2n$ matrices given by
\[
	\spp(W_\C) 
	=	\left\{ 
			\mtwo{a}{b}{c}{-\tp{a}} 
			: b = \tp{b}, c = \tp{c} 
		\right\}.
\]
Let $\h$ be the Cartan algebra consisting of diagonal matrices,
\[
	\h 
	=	\big\{ 
			a = \op{diag}( a_1, \dots, a_n, -a_1, \dots, -a_n ) 
			: a_i \in \C 
		\big\} 
	\cong \C^n.
\]
If 
	$\{
		\epsilon_1, \dots, \epsilon_n
	\}$ 
is the dual basis of $\h^\ast$, i.e., $\epsilon_i(a) = a_i$, then the set of roots is
\[
	\Delta 
	=	\big\{ 
			\pm ( \epsilon_i \pm \epsilon_j ) 
			: 1 \leq i < j \leq n 
		\big\}
	\cup
		\big\{ 
			\pm 2\epsilon_i 
			: 1 \leq i \leq n 
		\big\},
\]
and the simple roots may be taken to be 
	$\Pi 
	=	\{
			\alpha_1, \dots, \alpha_n
		\}$, 
where $\alpha_n = 2\epsilon_n$ and $\alpha_i = \epsilon_i - \epsilon_{i+1}$ for $i < n$.  The simple affine roots are 
	$\Pi^{\aff} 
	=	\{
			\alpha_0, \alpha_1, \dots, \alpha_n
		\}$, 
where $\alpha_0 = -\delta +1 = -2\epsilon_1 + 1$.  

The Weyl group $\Omega$ is generated by $s_{\alpha_1}, \dots, s_{\alpha_n}$, and the affine Weyl group $\Omega^{\aff}$ is generated by  $s_{\alpha_0}, \dots, s_{\alpha_n}$.  Both of these are Coxeter groups, and the braid relations are given by the extended Dynkin diagram of type C$_n$, cf.\ \cite[\S1.8]{iwahori-matsumoto}.
\begin{center}
\begin{tikzpicture}[scale=.5]
	\draw (0,-.1) --(2,-.1);
	\draw (0,.1) --(2,.1);
	\draw (2,0) --(5,0);
	\draw (7,0) --(8,0);
	\draw (8,-.1) --(10,-.1);
	\draw (8,.1) --(10,.1);
	\draw (9.2,.3) --(8.9,0) --(9.2,-.3);
	\draw (.8,.3) --(1.1,0) --(.8,-.3);
	\draw[fill] (0,0) circle(5pt);
	\draw[fill] (2,0) circle(5pt);
	\draw[fill] (4,0) circle(5pt);
	\draw[fill] (8,0) circle(5pt);
	\draw[fill] (10,0) circle(5pt);
	\node at (0,-.8) {$\alpha_0$};
	\node at (2,-.8) {$\alpha_1$};
	\node at (4,-.8) {$\alpha_2$};
	\node at (10,-.8) {$\alpha_n$};
	\node at (6,0) {$\cdots$};
\end{tikzpicture}
\end{center}

The fundamental chamber 
	$\mathcal{C}_0 \subset \h_\R \cong \R^n$ 
is bounded by the simple affine hyperplanes 
	$P_{\alpha_0}, \dots, P_{\alpha_n}$; 
its vertices are $z_0, \dots, z_n$, where $z_i$ is the intersection of the $P_{\alpha_j}$ such that $j \neq i$, i.e.,
\[
	z_0 
	=	(0, 0, \dots, 0), 
\quad 
	z_1 
	=	(\tfrac{1}{2}, 0, \dots, 0), 
\quad \dots, \quad 
	z_n 
	=	(\tfrac{1}{2}, \tfrac{1}{2}, \dots, \tfrac{1}{2}).
\]

%%%%%%%%%%
%%%%%%%%%%
\section{Symplectic groups}
\label{symplectic_groups}
%%%%%

Let $W = X + Y$ be a $2n$-dimensional symplectic vector space over $\kk$.  The symplectic group $\Sp(W)$ is the group of invertible transformations of $W$ that preserve $Q$.  Under the symplectic basis, $\Sp(W)$ is the subgroup of $\GL_{2n}(\kk)$ given by
\[
	\Sp(W) 
	=	\left\{
			\mtwo{a}{b}{c}{d}
			:\begin{array}{r}
				\tp{a} d - \tp{c} b = 1 \\
				\tp{a} c - \tp{c} a = 0 \\
				\tp{b} d - \tp{d} b = 0
			\end{array}
		\right\}.
\]
					
The symplectic group has a maximal torus, $H = (\kk^\times)^n$, where $h = (h_1, \dots, h_n)$ acts on the symplectic basis by 
	$h(e_i) 
	=	h_i e_i$ 
and 
	$h(f_i) 
	=	h_i^{-1} f_i$.  
This torus corresponds to the choice of root system $\Delta$ given in the previous section.  Define 
	$H(\oo)
	=	(\oo^\times)^n$.

For $1 \leq i \leq n$, define $X_i$ to be the span of $e_1, \dots, e_i$.  The stabilizer of the flag 
	$X_1 \subset X_2 \subset \cdots \subset X_n$ 
is a Borel subgroup of $\Sp(W)$ containing $H$; it corresponds to the choice of simple roots $\Pi$ given in the previous section.

For any $\alpha \in \Delta$, there is an injective homomorphism 
	$\Phi_\alpha: 
		\SL_2(\kk) \to \Sp(W)$, 
cf.\ \cite[\S3 Cor.\ 6]{steinberg}.  For $t \in \kk$, write
\[
	x_\alpha(t)
	=	\Phi_\alpha \left( \mtwo{1}{t}{0}{1} \right), 
\quad
	x_{-\alpha}(t) 
	=	\Phi_\alpha \left( \mtwo{1}{0}{t}{1} \right),
\]
and, for $t \in \kk^\times$, write
\begin{align*}
	w_\alpha(t) 
	&	=	x_\alpha(t) x_{-\alpha}(-t^{-1}) x_\alpha(t) 
		=	\Phi_\alpha \left( \mtwo{0}{t}{-t^{-1}}{0} \right),
	\\
	h_\alpha(t) 
	&	=	w_\alpha(t) w_\alpha(-1) 
		=	\Phi_\alpha \left( \mtwo{t}{0}{0}{t^{-1}} \right).
\end{align*}
	
These maps may be realized explicitly as follows.  If $E_{ij}$ is the $n \times n$ matrix with a 1 in the $ij$ position and 0 elsewhere, then the $x_\alpha(t)$ are given by
\begin{align*}
	x_{\epsilon_i - \epsilon_j}(t) 
	&	=	\mtwo{ 1 + t E_{ij}}{0}{0}{1 - t E_{ji} }, 
	\\
	x_{\epsilon_i + \epsilon_j}(t) 
	&	=	\mtwo{1}{ t (E_{ij} + E_{ji}) }{0}{1}, 
	\\
	x_{2\epsilon_i}(t) 
	&	=	\mtwo{1}{ t E_{ii} }{0}{1}.
\end{align*}
Hence, for 
	$\alpha
	=	\epsilon_i - \epsilon_j$, 
	$\Phi_\alpha \big( \SL_2(\kk) \big)$ 
acts as $\SL_2(\kk)$ on the subspaces
\[
	\kk e_i \oplus \kk e_j \subset X
				\quad \text{ and } \quad 
	\kk f_j \oplus \kk f_i \subset Y,
\] 
and, for 
	$\alpha
	=	2 \epsilon_i$, 
	$\Phi_\alpha \big( \SL_2(\kk) \big)$ 
acts as $\SL_2(\kk)$ on the subspace 
\[
	\kk e_i \oplus \kk f_i \subset W.
\]

For any additive subgroup $A \subset \kk$ and any $\alpha \in \Delta$, define the subgroup 	
\[
	\XX_\alpha(A) 
	=	\big\{
			x_\alpha(t) 
			: t \in A 
		\big\} 
	\subset \Sp(W).
\] 
The symplectic group is generated by the $\XX_\alpha(\kk)$, as $\alpha$ ranges over $\Delta$.

For an affine root $\alpha + m$, define 
	$\Phi_{\alpha+m}: 
	\SL_2(\kk) \to \Sp(W)$
by
\[
	\Phi_{\alpha+m} 
		\left( 
			\mtwo{a}{b}{c}{d} 
		\right)
	=	\Phi_{\alpha} 
		\left( 
			\mtwo{a}{\varpi^m b}{\varpi^{-m} c}{d} 
		\right).
\]
The subgroup 
	$\Phi_{\alpha+m} \big( \SL_2(\kk) \big)$
contains the elements
\begin{align*}
	x_{\alpha+m}(t) 
		=	\Phi_{\alpha+m}
				\left( 
					\mtwo{1}{t}{0}{1}
				\right) 
	&	=	x_\alpha( \varpi^m t),
				\quad t \in \kk, 
	\\
	w_{\alpha+m}(t)
		=	\Phi_{\alpha+m}
				\left( 
					\mtwo{0}{t}{-t^{-1}}{0} 
				\right) 
	&	=	w_\alpha( \varpi^m t),
				\quad t \in \kk^\times, 
	\\
	h_{\alpha+m}(t)
		=	\Phi_{\alpha+m}
				\left( 
					\mtwo{t}{0}{0}{t^{-1}} 
				\right) 
	&	=	h_\alpha( t), 
				\quad t \in \kk^\times,
\end{align*}
and is generated by $\XX_{\alpha+m}(\kk)$ and $\XX_{-(\alpha+m)}(\kk)$.	 For 
	$\alpha + m \in \Delta^{\aff}$, 
define the affine root group $\XX_{\alpha+m}$ to be $\XX_{\alpha}(\varpi^m \oo)$.

The torus $H$ is generated by 
	$\{ 
		h_\alpha(t) 
		: \alpha \in \Delta, t \in \kk^\times 
	\}$; 
the normalizer $N$ of the torus is generated by 
	$\{
		h_\alpha(t), w_\alpha(t) 
		: \alpha \in \Delta, t \in \kk^\times 
	\}$. 
The generators of $\Omega \cong N / H$ are represented in $\Sp(W)$ by 
\[
	\big\{
		w_{\alpha_1}(1), \dots, w_{\alpha_n}(1) 
	\big\} 
	= \big\{ 
		w_\alpha(1) 
		: \alpha \in \Pi 
	\big\},
\]
and the generators of $\Omega^{\aff} \cong N / H(\oo)$ are represented by 
\[
	\big\{
		w_{\alpha_0}(1), \dots, w_{\alpha_n}(1) 
	\big\} 
	= \big\{ 
		w_\mu(1) 
		: \mu \in \Pi^{\aff} 
	\big\},
\]
cf.\ \cite[Lemma 22]{steinberg}.  The longest element 
	$s_{2\epsilon_1} \cdots s_{2\epsilon_n}$ 
of the Weyl group is represented in $\Sp(W)$ by
\[
	w = \mtwo{0}{1}{-1}{0}.
\]
	
For an element $d = (d_1, \dots, d_n)$ of the co-root lattice, $T(d)$ is an element of the affine Weyl group.  Under the identification of $s_{\alpha_i}$ and $w_{\alpha_i}(1)$, by \cite[\S2.1]{iwahori-matsumoto}, $T(d)$ may be represented in $\Sp(W)$ as the toral element
\[
	h_{2\epsilon_1}(\varpi^{-d_1}) \cdots \h_{2\epsilon_n}(\varpi^{-d_n}) 
	=	(\varpi^{-d_1}, \dots, \varpi^{-d_n}).
\]

For $0 \leq i \leq n$, define 
	$\eta_i 
	=	h_{2\epsilon_1}(\varpi^{-1}) \cdots h_{2\epsilon_i}(\varpi^{-1})$,
which corresponds to 
	$T(1, \dots, 1, 0, \dots, 0)$, 
where each of the first $i$ components is 1.  The element $\eta_i w$ is the longest element of the affine Weyl group that stabilizes the point $z_i$.

%%%%%%%%%%
%%%%%%%%%%
\section{Compact subgroups}
\label{compact_subgroups}
%%%%%

The fundamental chamber corresponds to an Iwahori subgroup $I$ of $\Sp(W)$ in the sense that its unipotent radical is generated by those affine root groups $\XX_\mu$ for which $\mu$ acts positively on $\mathcal{C}_0$; explicitly, $I$ is the group generated by 
	$\XX_\alpha$ for $\alpha \in \Delta^+$, 
	$\XX_{\alpha+1}$ for $\alpha \in \Delta^-$, 
and $H(\oo)$.

The maximal compact subgroups of $\Sp(W)$ containing $I$ are obtained by choosing a maximal proper subset 
	$\Pi^{\aff} \smallsetminus \{ \alpha_i \}$
of the simple affine roots, or, equivalently, by choosing a vertex $z_i$ of the fundamental chamber of $\Delta^{\aff}$. The vertex $z_i$ corresponds to the maximal compact subgroup $K_i$, where $K_i$ is the group generated by those $\XX_\mu$ for which $\mu(z_i) \geq 0$, cf.\ \cite[\S2.5]{iwahori-matsumoto}. 

Removing the simple affine root $\alpha_i$ from the extended Dynkin diagram yields
\begin{center}
\begin{tikzpicture}[scale=.5]
	\draw (2,0) --(5,0);
	\draw (7,0) --(8,0);
	\draw (8,-.1) --(10,-.1);
	\draw (8,.1) --(10,.1);
	\draw (9.2,.3) --(8.9,0) --(9.2,-.3);
	\draw (-.15,-.15) --(.15,.15); \draw (-.15,.15) --(.15,-.15);
	\draw[fill] (2,0) circle(5pt);
	\draw[fill] (4,0) circle(5pt);
	\draw[fill] (8,0) circle(5pt);
	\draw[fill] (10,0) circle(5pt);
	\node at (0,-.8) {$\alpha_0$};
	\node at (2,-.8) {$\alpha_1$};
	\node at (4,-.8) {$\alpha_2$};
	\node at (10,-.8) {$\alpha_n$};
	\node at (6,0) {$\cdots$};
	\node at (16,0) {if $i=0$,};
	\node at (-3,0) {\mbox{ }};
	\node at (18,0) {\mbox{ }};
\end{tikzpicture}
\end{center}

\begin{center}
\begin{tikzpicture}[scale=.5]
	\draw (0,-.1) --(2,-.1);
	\draw (0,.1) --(2,.1);
	\draw (.8,.3) --(1.1,0) --(.8,-.3);
	\draw (2,0) --(3,0);
	\draw (5,0) --(6,0);
	\draw (10,0) --(11,0);
	\draw (13,0) --(14,0);
	\draw (14,-.1) --(16,-.1);
	\draw (14,.1) --(16,.1);
	\draw (15.2,.3) --(14.9,0) --(15.2,-.3);
	\draw[fill] (0,0) circle(5pt);
	\draw[fill] (2,0) circle(5pt);
	\draw[fill] (6,0) circle(5pt);
	\draw[fill] (10,0) circle(5pt);
	\draw[fill] (14,0) circle(5pt);
	\draw[fill] (16,0) circle(5pt);
	\draw (7.85,-.15) --(8.15,.15); \draw (7.85,.15) --(8.15,-.15);
	\node at (0,-.8) {$\alpha_0$};
	\node at (2,-.8) {$\alpha_1$};
	\node at (4,0) {$\cdots$};
	\node at (6,-.8) {$\alpha_{i-1}$};
	\node at (8,-.8) {$\alpha_i$};
	\node at (10,-.8) {$\alpha_{i+1}$};
	\node at (12,0) {$\cdots$};
	\node at (14,-.8) {$\alpha_{n-1}$};
	\node at (16,-.8) {$\alpha_n$};
	\node at (20,0) {if $0 < i < n$,};
\end{tikzpicture}
\end{center}

\begin{center}
\begin{tikzpicture}[scale=.5]
	\draw (0,-.1) --(2,-.1);
	\draw (0,.1) --(2,.1);
	\draw (.8,.3) --(1.1,0) --(.8,-.3);
	\draw (2,0) --(5,0);
	\draw (7,0) --(8,0);
	\draw[fill] (0,0) circle(5pt);
	\draw[fill] (2,0) circle(5pt);
	\draw[fill] (4,0) circle(5pt);
	\draw[fill] (8,0) circle(5pt);
	\draw (9.85,-.15) --(10.15,.15); \draw (9.85,.15) --(10.15,-.15);
	\node at (0,-.8) {$\alpha_0$};
	\node at (2,-.8) {$\alpha_1$};
	\node at (4,-.8) {$\alpha_2$};
	\node at (10,-.8) {$\alpha_n$};
	\node at (6,0) {$\cdots$};
	\node at (16.1,0) {if $i = n$,};
	\node at (-3,0) {\mbox{ }};
	\node at (18,0) {\mbox{ }};
\end{tikzpicture}
\end{center}
so $K_i$ is generated by the 
	$\XX_{\alpha_j}$ ($j \neq i$) 
and 
	$I \cap \Phi_{\alpha_i}\big(\SL_2(\kk)\big)$; 
in other words, $K_i$ is generated by the 
	$\XX_{\alpha_j}$ ($0 \leq j \leq n$) 
along with
	$\XX_{-\alpha_i + 1}$ and $H(\oo)$.

For $0 \leq i \leq n$, it is straightforward to verify that $K_i$ is the stabilizer of the lattice $\LL_i$ and its dual $\LL_i^\ast$, where
\[
	\LL_i 
	=	\big[
			\oo e_1 \oplus \cdots \oplus \oo e_n 
		\big]
	\oplus
		\big[
			\varpi( \oo f_1 \oplus \cdots \oplus \oo f_i )
			\oplus \oo f_{i+1} \oplus \cdots \oplus \oo f_n 
		\big].
\]
Note that $\LL_0$ is the standard lattice, i.e., the $\oo$-span of the symplectic basis, and that
\[
	\LL_n \subset \cdots \subset \LL_1 \subset \LL_0.
\]
	
The element $\eta_i w$ stabilizes $\LL_i$, and is hence an element of $K_i$.  In fact, 
\[
	\eta_i w\XX_{-\alpha_i+1} (\eta_i w)^{-1} 
	=	\XX_{\alpha_i},
\]
so $K_i$ is generated by $I$ and $\eta_i w$.

Note that the stabilizer of $\LL_0$ is 
	$K_0 
	=	\Sp(W) \cap \GL_{2n}(\oo) 
	=	\Sp_{2n}(\oo)$, 
while the stabilizer of $\LL_n$ is 
	$K_n 
	=	g^{-1} K_0 g$, 
where $g \in \GL_{2n}(\kk)$ is given by
\[
	g = \mtwo{ \varpi 1 }{0}{0}{1}.
\]
For $0 < i < n$, $K_i$ contains a proper subgroup
\[
	\langle 
		\XX_{\alpha_j} 
		: j < i 
	\rangle 
		\times 
	\langle 
		\XX_{\alpha_j} 
		: j > i 
	\rangle
		\cong 
	\Sp_{2i}(\oo) \times \Sp_{2(n-i)}(\oo).
\]
The first factor is a 
	`subgroup of type $K_n$' of $\Sp_{2i}(\kk)$; 
the second factor is a 
	`subgroup of type $K_0$' of $\Sp_{2(n-i)}(\kk)$.

%%%%%%%%%%
%%%%%%%%%%
\section{Weil representation}
\label{weil_representation}
%%%%%

The Heisenberg group $H(W)$ is defined to be the set $W \times \kk$ with group multiplication
\[
	(u, s) \cdot (v, t) 
	=	\big( 
			u + v, 
			s + t + Q(u,v) 
		\big).
\]
The symplectic group $\Sp(W)$ preserves $Q$, hence it acts as a group of automorphisms on $H(W)$ by $g(v, t) = (gv, t)$.

Let $(\rho, S)$ be a representation of $H(W)$ with central character $\psi$; that is, $\rho$ acts by $\psi$ on the center 
	$\{ 0 \} \times \kk \cong \kk$.  
For $g\in \Sp(W)$ the twist of $\rho$ by $g$, given by 
	$\rho^g(v,t) 
	=	\rho(gv, t)$, 
also has central character $\psi$.  The Stone-von Neumann theorem implies that $\rho^g$ is isomorphic to $\rho$, so there exists an intertwining operator $T(g)$ on $S$ such that 
	$T(g) \rho
	=	\rho^g T(g)$.  
This operator is unique up to scalar and thus defines a projective representation $(T, S)$ of $\Sp(W)$, which lifts uniquely to a linear representation $(\omega, S)$ of the two-fold central extension $\ce\Sp(W)$ of $\Sp(W)$, called the Weil representation with respect to $\psi$, cf.\ \cite{weil}.

%\subsection{Lattice model}

%%%%%%%%%%
%%%%%%%%%%
\subsection{Schr\"odinger's model} 
\label{schrodinger_model}
%%%%%

The polarization $W = X + Y$ gives a realization for $\omega$, called the Schr\"odinger model, on the space $S(Y)$ of Schwartz functions on $Y$, 
\begin{align*}
	\big[ \xx(a) \phi \big](y) 
	&	=	\psi( \tp{y} a y ) \phi(y),
	\\
	\big[ \hh(a) \phi \big](y) 
	&	=	\beta_a | \det a |^{\fr{1}{2}} \phi( \tp{a} y ), 
	\\
	\big[ \ww \phi \big](y) 
	&	=	\gamma_1 \widehat{\phi}(y).
\end{align*}
Here,
\begin{align*}
	\xx(a) \text{ is a lift of } 
	&	\mtwo{1}{a}{0}{1} 
		\text{ for } a \in M_n(\kk) \text{ with } \tp{a} = a, 
	\\
	\hh(a) \text{ is a lift of } 
	&	\mtwo{a}{0}{0}{ \tpinv{a} }
		\text{ for } a \in \GL_n(\kk), 
	\\
	\ww \text{ is a lift of } w = 
	&	\mtwo{0}{1}{-1}{0}.
\end{align*}
See \cite[\S1.7]{wood} for the details of this particular description.  The constants $\beta_a$ and $\gamma_1$ are complex roots of unity, but their specific values are, at present, of no consequence.  Moreover, as the specific values of $\beta_a$, $\det a$, and $\gamma_1$ will play no role in this paper, the letter $c$ will frequently be used in their stead.  The exception to this is the following fact: if $\det a = 1$, then $\beta_a = 1$.

As $\ce\Sp(W)$ is a central extension of $\Sp(W)$, the elements $x_\alpha(t) \in \Sp(W)$ lift canonically to 
	$\xx_\alpha(t) \in \ce\Sp(W)$; 
the $\xx_\alpha(t)$ are additive in $t$, so each affine root group $\XX_\mu$ lifts uniquely to $\ce\XX_\mu$, cf.\ \cite[Thm 10]{steinberg}.  Let $\HH$ and $\HH(\oo)$ be the full inverse images of $H$ and $H(\oo)$, respectively.

%%%%%
\begin{lemma} 
\label{chevalley_action}
The Chevalley generators $\xx_\alpha(t)$ act on $S(Y)$ by
\begin{align*}
	\big[ \xx_{\epsilon_j - \epsilon_k}(t) \phi \big](y) 
	&	=	\phi( y + t y_j f_k ), 
	\\
	\big[ \xx_{\epsilon_j + \epsilon_k}(t) \phi \big](y) 
	&	=	\psi( 2 t y_j y_k ) \phi(y), 
	\\
	\big[ \xx_{2\epsilon_j}(t) \phi \big](y) 
	&	=	\psi(t y_j^2) \phi(y).
\end{align*}
A lift $\hh \in \HH$ of $h = (h_1, \dots, h_n) \in H$ acts by
\[
	\big[ \hh \phi \big](y) 
	=	c \phi\big( h^{-1}(y) \big) 
	=	c \phi( h_1 y_1, \dots, h_n y_n ).
\]
\end{lemma}
%%%%%%%%%%

%%%%%%%%%%
%%%%%%%%%%
\section
{$\ce K$-types in Schr\"odinger's model} 
\label{K-types}
%%%%%

The maximal compact subgroups of $\ce\Sp(W)$ are the full inverse images of those of $\Sp(W)$; denote by $\KK_i$ the full inverse image of $K_i$.  Recall that $K_i$ is the stabilizer of the lattice
\[
	\LL_i 
	=	\big[ 
			\oo e_1 \oplus \cdots \oplus \oo e_n 
		\big]
	\oplus
		\big[ 
			\varpi ( \oo f_1 \oplus \cdots \oplus \oo f_i ) 
			\oplus \oo f_{i+1} \oplus \cdots \oplus \oo f_n 
		\big];
\]
it is also the stabilizer of the dual lattice
\[
	\LL_i^\ast 
	=	\big[ 
			\varpi^{-1} ( \oo e_1 \oplus \cdots \oplus \oo e_i )
			\oplus \oo e_{i+1} \oplus \cdots \oplus e_n 
		\big]
	\oplus
		\big[ 
			\oo f_1 \oplus \cdots \oplus f_n 
		\big].
\]
Define
\[
	L_i 
	=	\LL_i \cap Y 
	=	\varpi \oo^i \oplus \oo^{n-i}
				\quad \text{ and } \quad
	L'_i 
	=	\LL_i^\ast \cap Y 
	=	\oo^n.
\]
Then,
\[
	\cdots \subset
		\varpi L_i \subset 
		\varpi L'_i \subset
		L_i \subset 
		L'_i \subset
		\varpi^{-1} L_i \subset 
		\varpi^{-1} L'_i \subset 
	\cdots.
\]

Suppose that $\phi \in S(Y)$ is invariant under translation by 
\[
	2 \varpi^m L_i 
	=	2 \varpi^m( \varpi \oo^i \oplus \oo^{n-i} )
\]  
so that $\widehat\phi$, by Lemma \ref{fourier_transform_space}, is supported on 
	$\varpi^{-m}( \varpi^{-1} \oo^i \oplus \oo^{n-i} )$.  
If $\ce \eta_i$ is a lift of $\eta_i$, then $\ce\eta_i \ww \phi$ is supported on $y \in Y$ such that
\[
	\eta_i (y) \in \varpi^{-m} ( \varpi^{-1} \oo^i \oplus \oo^{n-i} )
			\quad \iff \quad
	y \in \varpi^{-m} L'_i.
\]
Similarly, if $\phi$ is invariant under translation by $2 \varpi^m L'_i$, then $\ce\eta_i \ww \phi$ is supported on $\varpi^{-m} L_i$.  Therefore, since $\ce \eta_i \ww \in \KK_i$,
\[
	S_{i,m}
	=	S\big( \varpi^{-m} L'_i / 2 \varpi^m L_i \big)
					\quad \text{ and } \quad
	S'_{i,m} 
	=	S\big( \varpi^{-m} L_i / 2 \varpi^m L'_i \big)
\]
are candidate spaces for the action of $\KK_i$.  There are, of course, restrictions on the permissible values of $m$: $S_{i,m}$ only makes sense for $2m \geq -e - 1$ while $S'_{i,m}$ only makes sense for $2m \geq -e$.

%%%%%
\begin{lemma} \label{geometric_action}
For $0 \leq i \leq n$, $\HH(\oo)$ and $\ce\XX_{\epsilon_j - \epsilon_k}$ preserve $S_{i,m}$ and $S'_{i,m}$.
\end{lemma}

\begin{proof}
The action of these elements is given by Lemma \ref{chevalley_action}.  For $h \in (\oo^\times)^n$ and $t \in \oo$, the maps
\[		
	y \mapsto h^{-1}(y)
		\quad \text{ and } \quad
	y \mapsto y + t y_j f_k		
\]
preserve the quotients 
\begin{align*}
	\varpi^{-m} L'_i / 2 \varpi^m L_i 
	&	=	\oo^i / 2 \varpi \oo^i \oplus \oo^{n-i} / 2 \oo^{n-i},
	\\
	\varpi^{-m} L_i / 2 \varpi^m L'_i
	&	=	\varpi \oo^i / 2 \oo^i \oplus \oo^{n-i} / 2 \oo^{n-i},
\end{align*}
hence the lemma.
\end{proof}
%%%%%

The following lemma and corollary will play a role in computing certain eigenspaces in the Weil representation.

%%%%%
\begin{lemma}
\label{boring_lemma}
Fix $r,s \in \oo$ and $m \in \Z$.  If $\psi(tr^2) = \psi(ts^2)$ for all $t \in \varpi^m \oo$, then $r \equiv \pm s$ modulo $2 \varpi^{-\lfloor \fr{m}{2} \rfloor} \oo$.
\end{lemma}

\begin{proof}
Since the conductor of $\psi$ is $2e$, the hypothesis
	$\psi \big( t (r-s) (r+s) \big) = 1$
implies that
\[
	2 \max \{ \val(r-s), \val(r+s) \}
		\ge \val(r-s) + \val(r+s)
		\ge 2e - m.
\]
As $\val(r \pm s)$ is an integer, one of $r-s$ or $r+s$ is in 
$2 \varpi^{-\lfloor \fr{m}{2} \rfloor} \oo$.
\end{proof}
%%%%%

%%%%%
\begin{corollary}
\label{characters_of_sym}
For $x \in \kk^n$ and $m \in \Z$, define the character $\psi_{x,m}$ on 
	$\op{Sym}_n ( \varpi^m \oo )$,
the additive group of symmetric $n \times n$ matrices with entries in $\varpi^m \oo$, by
\[	
	\psi_{x,m}(a)
	=	\psi( \tp{x} a x ).
\]
For $x,y \in \kk^n$, the following conditions hold.
\begin{enumerate}[ \mbox{ } 1.]
	\item 
	If $m \le 1$ and $\psi_{x,m} = \psi_{y,m}$, 
	then $x$ is congruent to $y$ modulo $2 \oo^n$.
	\item
	If $m \le -1$ and $\psi_{x,m} = \psi_{y,m}$, 
	then $x$ is congruent to $\pm y$ modulo $2 \varpi \oo^n$.
\end{enumerate}
\end{corollary}

\begin{proof}
The condition $\psi_{x,m} = \psi_{y,m}$ translates to
\[
	\psi( \tp{x} a x ) = \psi( \tp{y} a y )
			\quad \text{ for all } \quad
		a \in \op{Sym}_n( \varpi^m \oo ).
\]
Taking $a = t E_{ii}$, this gives $\psi(t x_i^2) = \psi(t y_i^2)$ for all $t \in \varpi^m \oo$, hence $x_i \equiv \pm y_i$ modulo $2 \varpi^{-\lfloor \fr{m}{2} \rfloor} \oo$ by Lemma \ref{boring_lemma}.

In the first case where $m \le 1$, one has $x_i \equiv \pm y_i$ modulo $2 \oo$, hence $x_i \equiv y_i$.  As this is true for each $i$, it must hold that $x \equiv y$ modulo $2 \oo^n$.

In the second case where $m \le -1$, one has $x_i \equiv \pm y_i$ modulo $2 \varpi \oo$ for each $i$, so it remains to verify that the sign is independent of the index.  If this is not the case, then there exist indeces $i, j$ such that
\[
	-y_i y_j \equiv x_i x_j \not\equiv -x_i x_j \text{ modulo } 2 \varpi \oo,
\]
hence $\val(x_i x_j) = 0$.  Taking $a = t (E_{ij} + E_{ji})$ for  $t \in \varpi^m \oo$, the original hypothesis implies that
\[
	\psi( 2 t x_i x_j ) = \psi( 2 t y_i y_j ).
\]
The congruence above then gives that
\[
	\psi( 4 t x_i x_j ) = 1
			\quad \text{ for all } \quad
		t \in \varpi^m \oo,
\]
which is a contradiction since $4 t x_i x_j$ is not necessarily in $4 \oo$.
\end{proof}
%%%%%

%%%%%%%%%%
%%%%%%%%%%
\subsection{$\KK_0$-types} 
\label{K_0-types}
%%%%%

Since $L'_0 = L_0 = \oo^n$, the chain of lattices is
\[
	\cdots \subset
	\varpi L_0 
	=	\varpi L'_0 \subset
	L_0 
	=	L'_0 \subset
	\varpi^{-1} L_0 
	=	\varpi^{-1} L'_0 \subset 
	\cdots,
\]
hence
\[
	S'_{0,m} 
	=	S_{0,m} 
	=	S( \varpi^{-m} L_0 / 2 \varpi^m L_0 ) 
	=	S( \varpi^{-m} \oo^n / 2 \varpi^m \oo ).
\]

The main result of this section is that $\KK_0$ acts irreducibly on  $S_{0,0} = S( L_0 / 2L_0 )$;  this space contains no odd functions and $\{ \phi_x: x \in L_0 / 2L_0 \}$ is a basis, where $\phi_x$ is the characterstic function $\ch(x + 2 L_0)$.

Recall that $\KK_0$ is generated by the $\ce \XX_{\alpha_j}$ ($0 \leq j \leq n$) along with $\ww$ and $H(\oo)$.  Moreover, $\xx(a)$ is in $\KK_0$ if and only if $a \in M_n(\oo)$ with $\tp{a} = a$.

%%%%%
\begin{theorem} \label{K_0-theorem}
The group $\KK_0$ preserves the filtration
\[
	S_{0,0} \subset S_{0,1} \subset S_{0,2} \subset \dots
\]
and acts irreducibly on 
	$S_{0,0} = S( L_0 / 2 L_0 ) = S( \oo^n / 2 \oo^n )$.
\end{theorem}

\begin{proof}
Fix $m \geq 0$.  By Lemma \ref{geometric_action}, to prove that $\KK_0$ preserves $S_{0,m}$, it remains to consider the action of 
	$\ce\XX_{\alpha_n} = \ce \XX_{2\epsilon_n}$ 
and 
	$\ww^{-1} \ce\XX_{\alpha_0} \ww = \ce\XX_{-\alpha_0+1} = \ce\XX_{2\epsilon_1}$.  
For any $1 \leq j \leq n$,
\[
	\big[ \xx_{2\epsilon_j}(t) \phi \big] (y)
	=	\psi(t y_j^2) \phi(y),
\]
so $\phi$ and $\xx_{2\epsilon_j}(t) \phi$ have the same support.  If $t \in \oo$ and $y_j \in \varpi^{-m} \oo$, then both $2 t y_j (2 \varpi^m \oo)$ and $t (2 \varpi^m \oo)^2$ are subsets of $4 \oo$; hence,
\[
	\psi \big( t (y_j + 2 \varpi^m \oo)^2 \big) 
	=	\psi( t y_j^2 ).
\]  
Therefore, if $\phi$ is invariant under translation by $2 \varpi^m L_0$, then so is $\xx_{2\epsilon_j}(t) \phi$, proving that $\KK_0$ indeed acts on $S_{0,m}$.

Finally, the question of irreducibility is considered.  For $x \in L_0 / 2L_0$ and $a \in \op{Sym}_n(\oo)$, $\xx(a)$ acts on $\phi_x$ by the character $\psi_{x,0}$ as in Corollary \ref{characters_of_sym}.  This corollary implies that the characters $\psi_{x,0}$ ($x \in L_0 / 2L_0$) are distinct, so an irreducible component of $S_{0,0}$ must be a direct sum of lines of the form $\C \phi_x$.  Similarly, $\ww^{-1} \xx(a) \ww$ acts on $\widehat\phi_x$ by the character $\psi_{x,0}$, so an irreducible component is a direct sum of lines $\C \widehat\phi_x$.  The component containing 
\[
	\widehat\phi_0
	=	\sum_{x \in L_0 / 2L_0} \phi_x
\]
must therefore be all of $S_{0,0}$.
\end{proof}
%%%%%

%%%%%
\begin{remark}
The proof of the irreducibility of the action restricted to $\KK_0$ on $S_{0,0}$ holds as well for the the action restricted to $\II$.  Indeed, since $\II$ contains $\ww^{-1} \xx(a) \ww$ for $a \in \op{Sym}_n(\varpi \oo)$, Corollary \ref{characters_of_sym} still implies that the corresponding eigenspaces are $\C \widehat\phi_x$ ($x \in L_0 / 2 L_0$).
\end{remark}
%%%%%

%%%%%%%%%%
%%%%%%%%%%
\subsection{$\KK_n$-types} \label{K_n-types}
%%%%%

Since $\varpi L'_n = L_n = \varpi \oo^n$, the chain of lattices is
\[
	\cdots \subset
	\varpi^2 L'_n = \varpi L_n \subset
	\varpi L'_n = L_n \subset
	L'_n = \varpi^{-1} L_n \subset \cdots,
\]
hence
\[
	S'_{n,m+1} 
	=	S_{n,m} 
	=	S( \varpi^{-m} L'_n / 2 \varpi^m L_n ) 
	=	S( \varpi^{-m} \oo^n / 2 \varpi^{m+1} \oo^n ).
\]
				
The main result of this section is that $\KK_n$ acts irreducibly on the even and odd components of 
\[
	S_{n,0} 
	=	S( L'_n / 2 L_n ) 
	=	S( \oo^n / 2 \varpi \oo^n);
\]
a basis for $S_{n,0}^\pm$ is 
	$\{ 
		\phi_x^{\pm} 
		: x \in L'_n / 2 L_n \; (\text{modulo } \pm 1) 
	\}$,
where 
\begin{align*}
	\phi_x^+ 
	&	=	\begin{cases} 
				\ch(x + 2 L_n) 
				& \text{ if } x \in L_n,
				\\
				\ch(x + 2 L_n) + \ch(-x + 2 L_n) 
				& \text{ if } x \notin L_n,
			\end{cases}
	\\
	\phi_x^-
	&	=	\ch(x + 2 L_n) - \ch(-x + 2 L_n). 
\end{align*}

Recall that $\KK_n$ is generated the $\ce \XX_{\alpha_j}$ ($0 \leq j \leq n$) along with $H(\oo)$ and $\ce\eta_n \ww$.  Moreover, $\xx(a)$ is in $\KK_n$ if and only if $a \in M_n(\varpi^{-1} \oo)$ with $\tp{a} = a$.

%%%%%
\begin{theorem}
\label{K_n-theorem}
The group $\KK_n$ preserves the filtration
\[
	S_{n,0} \subset 
	S_{n,1} \subset 
	S_{n,2} \subset \dots
\]
and acts irreducibly on 
	$S_{n,0}^\pm 
	=	S( L'_n / 2 L_n )^\pm 
	=	S( \oo^n / 2 \varpi \oo^n )^\pm$.
\end{theorem}

\begin{proof}
Fix $m \geq 0$.  By Lemma \ref{geometric_action} and the comment preceding it, to see that $\KK_n$ preserves $S_{n,m}$, it remains to consider the action of $\ce\XX_{\alpha_0}$ and $\ce\XX_{\alpha_n}$.  

Since
\[
	\ce\XX_{\alpha_0} 
	=	\ce\XX_{-2 \epsilon_1 + 1} 
	=	\ww^{-1} \ce\XX_{2 \epsilon_1 + 1} \ww,
\]
it must be shown that $\ce\XX_{2 \epsilon_1 + 1}$ preserves the Fourier transform of the space $S_{n,m}$, namely 
	$S( \varpi^{-m-1} \oo^n / 2 \varpi^m \oo^n )$.  
For $\phi \in S_{n,m}$,
\[
	\big[ \xx_{2 \epsilon_1}(t) \widehat\phi \big] (y)
	=	\psi( t y_1^2 ) \widehat\phi(y),
\]
so $\xx_{2 \epsilon_1}(t) \widehat\phi$ has the same support as $\widehat\phi$.  If $t \in \varpi \oo$ and $y \in \varpi^{-m-1} \oo^n$, then both 
	$2 t y_1 ( 2 \varpi^m \oo )$ 
and $t (2 \varpi^m \oo)^2$ are subsets of $4\oo$; hence,
\[
	\psi \big( t (y_1 + 2 \varpi^m \oo)^2 \big) 
	=	\psi( t y_1^2 ).
\]
Therefore, if $\widehat\phi$ is invariant under translation by $2 \varpi^m \oo^n$, then so is $\xx_{2 \epsilon_1}(t) \widehat\phi$.  

The proof that $\ce\XX_{\alpha_n}$ acts on $S_{n,m}$ follows from the proof of Theorem \ref{K_0-theorem} and the fact that $S_{0,m} \supset S_{n,m}$.

It remains to show that $\KK_n$ acts irreducibly on $S_{n,0}^{\pm}$.  For $x \in L_0 / 2 L_n$ and $a \in \op{Sym}_n(\varpi^{-1}\oo)$, the element $\xx(a) \in \KK_n$ acts on the line $\C \phi_x^\pm$ by the character $\psi_{x,-1}$ as in  Corollary \ref{characters_of_sym}.  By the corollary, since $2 L_n = 2 \varpi \oo^n$, these characters are distinct as $x$ varies over $L_0 / 2 L_n$.  Technically speaking, the corollary only gives distinctness as $x$ varies modulo $\pm 1$, but the equality $\C \phi_x^\pm = \C \phi_{-x}^\pm$ nullifies this issue.  Thus, an irreducible component of $S_{n,0}^{\pm}$ decomposes in terms of the $\C \phi_x^{\pm}$.  

By Lemma \ref{fourier_transform_space}, 
	\[		\widehat \phi_0^+
					=	\widehat{\ch(2 L_n)} 
					= \vol( 2 L_n) \ch( \varpi^{-1} L'_n ),		\] 
so $\ce\eta_n \ww \in \KK_n$ acts on $\phi_0^+$ by
	\begin{align*}
			\big[ \ce\eta_n \ww \phi_0^+ \big](y) 
					&= c \widehat \phi_0^+ \big( \eta_n(y) \big) 
					= c \widehat \phi_0^+ ( \varpi^{-1} y) 	\\
					&= c \vol(2 L_n) \sum_{ x \in L'_n / 2 L_n } \ch( x + 2 L_n ).
	\end{align*}
Therefore, the component of $S_{n,0}^+$ containing $\phi_0^+$ must be all of $S_{n,0}^+$.

Suppose now that $U$ is a nonzero irreducible component of $S_{n,0}^-$; that is, suppose that there exists $z \in \oo^n \smallsetminus \varpi \oo^n$ with $\C \phi_z^- \subset U$.  Each $\hh(a)$, $a \in \GL_n(\oo)$, acts (essentially) geometrically, so
	\[	\bigoplus_{z' \in \GL_n(\oo) z} \C \phi_{z'}^- \subset U.		\]

Let $x \in \oo^n \smallsetminus \varpi \oo^n$.  The action of $\ce\eta_n \ww$ on $\phi_x^-$ is given by
	\begin{align*}
		\big[ \ce\eta_n \ww \phi_x^-\big](y)
			& = c \widehat\phi_x^-(\varpi^{-1} y) \\
			& = c \big( \psi(2 \tp{x} y) - \psi(-2 \tp{x} y) \big) 
						\widehat{\phi}_0 (\varpi^{-1} y) \\
			& = c \big( \psi(2 \tp{x} y) - \psi(-2 \tp{x} y) \big) [ \ce\eta_n \ww \ch(2 L_n)](y).
	\end{align*}
Since $x = (x_1, \dots, x_n) \notin \varpi \oo^n$, there exists an index $j$ such that $x_j \in \oo^\times$.  Moreover, as $z \in \oo^n \smallsetminus \varpi \oo^n$, there exists $z' \in \GL_n(\oo)z$ such that the $j$th coordinate $z_j'$ is in $\oo^\times$ while the remaining coordinates are in $\varpi \oo$.  Then, $\tp{x} z' \in \oo^\times$, hence 
	\[		\psi( 2 \tp{x} z' ) \neq \psi( -2 \tp{x} z' ).		\]
Therefore, $\ce\eta_n \ww \phi_x^-$ is nonzero on $z'$, i.e., $\phi_x^-$ is in the same component as $\phi^-_{z'} \in U$.  As $x$ was arbitrary, the irreducible component $U$ must be all of $S_{n,0}^-$.
\end{proof}
%%%%%%%%%%%%%%%%%%%%%%%%%%%%%%%%%%%%%%

%%%%%%%%%%%%%%%%%%%%%%%%%%%%%%%%%%%%%%%%%%%%%%%%%%%%%%%%%%%%%%%%%%%%%%%%%%%%
%%%%%%%%%%%%%%%%%%%%%%%%%%%%%%%%%%%%%%%%%%%%%%%%%%%%%%%%%%%%%%%%%%%%%%%%%%%%
\subsection{$\KK_i$ types} \label{K_i-types}
%%%%%%%%%%%%%%%%%%%%%%%%%%%%%%%%%%%%%%

For $0 < i < n$, the spaces
	\[		S_{i,m} 
					= S( \varpi^{-m} L'_i / 2 \varpi^m L_i ),
			\quad \text{ } \quad
			S'_{i,m}
					= S( \varpi^{-m} L_i / 2 \varpi^m L'_i )			\]
are all distinct.  As $L_i = \varpi \oo^i \oplus \oo^{n-i}$ and $L'_i = \oo^n$, 
	\begin{align*}
		S_{i,m}
			&= S( \varpi^{-m} \oo^i / 2 \varpi^{m+1} \oo^i ) 
							\otimes S( \varpi^{-m} \oo^{n-i} / 2 \varpi^m \oo^{n-i} ), \\
		S'_{i,m}	
			&= S( \varpi^{-m+1} \oo^i / 2 \varpi^m \oo^i )
							\otimes S( \varpi^{-m} \oo^{n-i} / 2 \varpi^m \oo^{n-i} ).
	\end{align*}

%%%%%%%%%%%%%%%%%%%%%%%%%%%%%%%%%%%%%%
\begin{theorem} \label{K_i-theorem}
The group $\KK_i$ preserves the filtration
	\[		S_{i,0} \subset S'_{i,1} \subset S_{i,1} \subset S'_{i,2} \subset S_{i,2} \subset \dots		\]
and acts irreducibly on $S_{i,0}^{\pm} = S( L'_i / 2 L_i)^{\pm} = S( \oo^i / 2 \varpi \oo^i )^{\pm} \otimes S( \oo^{n-i} / 2 \oo^{n-i} )$.
\end{theorem}

\begin{proof}
It is first proved that $S_{i,m}$ ($m \geq 0$) and $S'_{i,m}$ ($m \geq 1$) are representations of $\KK_i$ by considering the generators $\HH(\oo)$, $\ce\XX_{-\alpha_i+1}$, and $\XX_{\alpha_j}$ ($0 \leq j \leq n$).

The group generated by the $\ce\XX_{\alpha_j}$ ($j < i$) is a subgroup of $\ce\Sp_{2i}(\kk)$ of `type $\KK_n$', which acts exclusively on the first component of the tensor product as in Theorem \ref{K_n-theorem}.  The group generated by the $\ce\XX_{\alpha_j}$ ($j > i$) is a subgroup of $\ce\Sp_{2(n-i)}(\kk)$ of `type $\KK_0$', which acts exclusively on the second component of the tensor product as in Theorem \ref{K_0-theorem}.   From the proofs of these theorems, it is seen that $\ce\XX_{\alpha_0}$ and $\ce\XX_{\alpha_n}$ preserve both $S_{i,m}$ and $S'_{i,m}$.  Lemma \ref{geometric_action} guarantees that $\HH(\oo)$ and the $\ce\XX_{\alpha_j}$ ($0 < j < n$) preserve these spaces as well.  Lastly, since
	\[		\ce\eta_i \ww \ce\XX_{\alpha_i} \ww^{-1} \ce\eta_i^{-1} = \ce\XX_{-\alpha_i+1},		\]
then $\ce\XX_{-\alpha_i + 1}$ must also preserve these spaces. Note that the restrictions on the values of $m$ arise from the proofs of Theorems \ref{K_0-theorem} and \ref{K_n-theorem}.

The irreducibility of $S_{i,0}^\pm$ follows from these same theorems since 
	\[		\langle \ce\XX_{\alpha_j} : 0 \leq j < i \rangle
		\quad \text{ and } \quad
			\langle \ce\XX_{\alpha_j} : i < j \leq n \rangle		\] 
act irreducibly on the respective tensor product components of $S_{i,0}^{\pm}$.
\end{proof}
%%%%%%%%%%%%%%%%%%%%%%%%%%%%%%%%%%%%%%

%%%%%%%%%%%%%%%%%%%%%%%%%%%%%%%%%%%%%%
%%%%%%%%%%%%%%%%%%%%%%%%%%%%%%%%%%%%%%
%%%%%%%%%%%%%%%%%%%%%%%%%%%%%%%%%%%%%%
%%%%%%%%%%%%%%%%%%%%%%%%%%%%%%%%%%%%%%

\mbox{}\newline
\noindent
Gordan Savin, 
{\sc 
Deptartment of Mathematics, \newline
University of Utah, 
Salt Lake City, UT 84112 } \newline
Email: {\tt savin@math.utah.edu}\newline

\noindent
Aaron Wood,
{\sc 
Department of Mathematics, \newline
University of Missouri, 
Columbia, MO 65201 } \newline
Email: {\tt woodad@missouri.edu}

\end{document}